\documentclass[11pt]{article}
\usepackage{amssymb,amsfonts,amsmath,amsthm,enumerate}
\usepackage[alphabetic]{amsrefs}

\def \al{\alpha}
\def \bs{\backslash}
\def \C{{\mathbb C}}
\def \CB{\mathcal{B}}
\def \CO{\mathcal{O}}

\def \e{\emph}
\def \Eig{\operatorname{Eig}}
\def \F{\mathbb F}
\def \Ga{\Gamma}
\def \ga{\gamma}
\def \GL{\operatorname{GL}}
\def \Hom{\operatorname{Hom}}

\def \la{\lambda}
\def \N{{\mathbb N}}
\def \ol{\overline}
\def \PGL{\operatorname{PGL}}
\def \PGSP{\operatorname{PGSP}}
\def \Per{\operatorname{Per}}
\def \Q{{\mathbb Q}}
\def \R{{\mathbb R}}

\def \T{{\mathbb T}}
\def \tr{\operatorname{tr}}

\def \what{\widehat}
\def \Z{{\mathbb Z}}
\def \({\left(}
\def \){\right)}

\newtheorem{thm}{Theorem}[section]
\newtheorem{theorem}[thm]{Theorem}

\newtheorem{lemma}[thm]{Lemma}

\newtheorem{proposition}[thm]{Proposition}
\numberwithin{equation}{section} \theoremstyle{definition}

\newtheorem*{remark}{Remark}

\renewcommand\sp[1]{\langle #1\rangle}
\newcommand{\norm}[1]{\left|\left| #1\right|\right|}
\newcommand{\stack}[2]{\genfrac{}{}{0pt}{}{#1}{#2}}

\begin{document}

\pagestyle{myheadings} \markright{ZETA FUNCTIONS OF F1-BUILDINGS}

\title{Zeta Functions of $\F_1$-buildings\\ \ \\
\small J. Math. Soc. Japan. Volume 68, Number 2, 807-822 (2016)}
\author{Anton Deitmar \& Ming-Hsuan Kang\thanks{This research was supported by NSC grant 100-2115-M-009-008-MY2 (M-H. Kang).
The research was performed
while the first author was visiting the Shing-Tung Yau Center in National Chiao Tung University and 
the National Center for Theoretical
Sciences, Mathematics Division, in Hsinchu, Taiwan. The authors would like
to thank both institutions for their support and hospitality.}}
\date{}
\maketitle

{\bf Abstract.} The analogue of the Bruhat-Tits building of a p-adic group in F1-geometry is a single apartment. In this setting, the trace formula gives rise to a several variable zeta function analogously to the p-adic case. The analogy carries on to the fact that the restriction to certain lines yield zeta functions which are defined in geometrical terms. Also, the classical formula of Ihara has an analogue in this setting.

{\bf Keywords:} field of one element, Bruhat-Tits building, Ihara zeta function
 
{\bf MSC:} {\bf 11F70}, 11F25, 11F66, 14G10, 20E42 

\section*{Introduction}
It has been observed by several authors that many formulae in the geometry of Bruhat-Tits buildings over a non-archimedean field $F$ of residue cardinality $q$ do still make sense for $q=1$, in which case they coincide with the analogous formulae on the corresponding spherical geometry.
Jacques Tits asked  in \cite{Tits}, whether the explanation of this phenomenon might be the existence of a ``field of one element'' $\F_1$ such that for a Chevalley group $G$ the group $G(\F_1)$ equals the Weyl group of $G$. 
In the first decade of the new millenium, various approaches to the elusive ``field'' $\F_1$ have been suggested, see \cites{KOW,Soule,F1,Haran,Toen,CC}.

In the middle between the geometry of the full Bruhat-Tits building and the spherical geometry, which may be considered as the local geometry of an apartment, there is the geometry of a single apartment and its affine Weyl group.
If the field of one element is the analogue of the residue field of a non-archimedean field, then the geometry of a single apartment should correspond to a ``non-archimedean field in characteristic one''. In order to fix terms we shall write $\Q_1$.
It appears that the approach via monoids of \cite{F1} yields an explanation whereby $\Q_1$ is the infinite cyclic group and one can describe the building of $\PGL_n(\Q_1)$ in terms of lattices in perfect analogy to the case of a non-archimedean field.

Another strand of investigations, which is connected to $\F_1$-theory in this paper, is the theory of generalized Ihara zeta functions.
The Ihara zeta function for a finite graph is defined as an Euler product over closed cycles in the graph. Its reciprocal turns out to be a polynomial and can be expressed in terms of the characteristic function of the adjacency operator, this latter fact being known under the name \e{Ihara formula}. Due to the complicated simplicial structure of higher dimensional building, there is no general formulation or conjecture of Ihara formula for higher dimensional building. The only known cases are $\PGL_3$ and $\PGSP_4$, see \cite{KL,KangLi,FLW}.

In Section \ref{SecIhara1}, we express the original Ihara formula for $\PGL_2(F)$ as a connection between Ihara zeta function and Langlands $L$-function.
Such a reformulation suggests a possible way to obtain a generalized Ihara formula and it still make sense when $q=1$.
In section 2 we apply the philosophy of the field of one element \cite{F1} to buildings.
In the world of $\F_1$-rings there exists, next to $\F_1$, a universal $\F_1$-algebra $\Z_1$ and its quotient field $\Q_1$.
It turns out that $\PGL_n(\Q_1)$ coincides with the extended affine Weyl group attached to an apartment the Bruhat-Tits of $\PGL_n(\Q_p)$ for any  prime number $p$.
It therefore is suggestive to consider one such apartment as the ``building'' of the grouop $\PGL_n(\Q_1)$.
In section 3 we develop the trace formula for the group $\PGL_n(\Q_1)$ which involves, among other things, the determinantion of its unitary dual.
As an application, we get a several variable zeta function $S_\Ga(u)$ which contains all information on the building and its quotient by a subgroup $\Ga$ from a group-theoretical viewpoint.
In section 4 we give a geomertic interpretation of $S_\Ga(u)$ in terms of homotopy classes of closed geodesics.
At the end of the paper, we give a generalized Ihara formula for $\PGL_n(\Q_1)$. This can be a starting point of finding a generalized Ihara formula of $\PGL_n(F)$.

\section{Ihara Zeta functions and Langlands L-functions} \label{SecIhara1}
Let $F$ be a non-archimedean field with $q$ elements in its residue field. Let $\CO_F$ be the ring of integers of $F$ and let $\pi$ be a uniformizer.
Given a discrete torsion-free cocompact subgroup $\Ga$ of $G(F)=\PGL_2(F)$, the quotient of the 
Bruhat-Tits tree of $\PGL_2(F)$ by $\Gamma$ is a finite $(q+1)$-regular graph $X_\Ga$. The Ihara zeta function of $X_\Ga$ is defined by
$$ Z_{X_\Ga}(u)=\prod_{[c]}(1-u^{l(c)})^{-1},
$$
where the product runs over all equivalence classes of prime cycles $[c]$ and $l(c)$ denotes the length.
The Ihara formula asserts that when $|u|$ is small enough, the reciprocal of $Z_{X_\Ga}(u)$ converges to a polynomial so that
$$
Z_{X_\Ga}(u)^{-1}=\det(1-Au+q u^2)(1-u^2)^{-\chi},
$$
where $A$ is the adjacency operator and $\chi$ is the Euler-characteristic (which is always non-positive).
The adjacency operator $A$ is indeed a Hecke operator $K \left(\begin{smallmatrix} 1 & \\ & \pi \end{smallmatrix}\right)K$ and it acts on $K$-fixed vectors of $L^2(\Gamma \backslash G)$. Here $K=G(\CO_F)$ is a maximal compact subgroup of $G$.

On the other hand, recall that unramified irreducible representations $(\rho,V)$ of $G$ are parametrized by the Satake parameters $s_\rho$, which are semisimple conjugacy classes in the complex dual group of $G$ which is isomorphic to SL$_2(\C)$. The Langlands $L$-function associated to $\rho$ is given by 
$$ L(\rho, u) = \det( 1- s_\rho u)^{-1}.$$
Now for the representation $L^2(\Gamma \backslash G)$, we define its Langlands $L$-function as the product of the $L$-functions of its irreducible unramified subrepresentations. From the Satake isomorphism, it is easy to see that
$$ L\left(L^2(\Gamma \backslash G), q^{1/2}u \right) = \det(1-Au + qu^2)^{-1}.$$
Thus, we can rewrite Ihara formula as 
$$Z_{X_\Ga}(u)= (1-u^2)^{\chi}  L\left(L^2(\Gamma \backslash G), q^{1/2}u \right).$$

Note that the right hand side of the above formula can be canonically generalized to $\PGL_n$ and it remains to figure out the left hand side.

On the other hand, in the case that $G=\PGL_2(F)$, if we just naively let $q=1$, then the building of $G$(a $(q+1)$-regular tree) becomes a 2-regular tree, which is indeed an apartment of the building. Then we shall replace $G$ by its affine Weyl group $W$, which acts on the standard apartment as automorphisms. 
In this case, for a discrete torsion-free cocompact subgroup $\Ga$ of $W$, one can define Langlands $L$-function of $L^2(\Ga \backslash W)$ by a similar way such that the above Ihara formula still holds.
 Details of the above would be covered later for the case of $\PGL_n$.

\section{The building of $\PGL_n(\Q_1)$}\label{sec1}
According to the philosophy of \cite{F1} we will denote
the trivial monoid of one element by $\F_1=\{1\}$.
Further we write $\Z_1=\{1,\tau,\tau^2,\dots\}$ for the free monoid of one generator $\tau$ and $\Q_1=\{\dots,\tau^{-1},1,\tau,\dots\}$ for its quotient group.
A \e{module} of a given monoid $A$ is a set with an $A$-action.
The category of $A$-modules contains direct sums, these turn out to be disjoint unions of modules.
For a given natural number $n$ consider the $\Q_1$-module $V=V_n=\bigoplus_{j=1}^n\Q_1=\coprod_{j=1}^n\Q_1$.
A \e{lattice} in $V$ is a finitely generated $\Z_1$-submodule $L$ of $V$ with the property that $\Q_1 L=V$.
Two lattices $L,L'$ are \e{homothetic} if there exists $\al\in\Q_1$ with $L'=\al L$.

The group $\GL_n(\Q_1)$ is by definition the automorphism group of the $\Q_1$-module $V=\bigoplus_{j=1}^n\Q_1$.
Each such automorphism permutes the copies of $\Q_1$ that make up $V$ and multiplies the inhabitants of each copy by a scalar in $\Q_1$.
The structure of this group is 
$$
\GL_n(\Q_1)\cong \Q_1^n\rtimes \Per(n),
$$
where $\Per(n)$ denotes the permutation group in $n$ letters.
Its center is the subgroup $\GL_1(\Q_1)\cong\Q_1$ embedded diagonally.
The group $\PGL_n(\Q_1)$ is defined to be
$$
\PGL_n(\Q_1)=\GL_n(\Q_1)/\GL_1(\Q_1).
$$
One way to picture $\GL_n(\Q_1)$ is to consider all $n\times n$ matrices with exactly one non-zero entry in every row and column and this entry be in $\Q_1$.
Then $\PGL_n(\Q_1)$ consist of homothety classes of such matrices.

The building $\CB$ of $\PGL_n(\Q_1)$ is the $(n-1)$-dimensional building defined as follows.
The set of vertices is the set of all homothety classes of lattices in $V_n$.
For $1\le k\le n-1$, distinct vertices $[L_0],\dots,[L_k]$ form the vertices of a $k$-dimensional face if, after adjusting the order, one has representatives satisfying $L_0\supset L_1\supset\dots\supset L_k\supset\tau L_0$.
Besides the mere geometry of being a building, this lattice description of $\CB$ adds more features, like the order of the vertices of a face which is determined up to a cyclic permutation.
First of all, there is a standard chamber $C_0$ given by the vertices
\begin{align*}
L_0&=\sp{e_1,\dots e_{n-2},e_{n-1},e_n},\\
L_1&=\sp{e_1,\dots, e_{n-2},e_{n-1},\tau e_n},\\
L_2&=\sp{e_1,\dots, e_{n-2},\tau e_{n-1},\tau e_n},\\
&\vdots\\
L_n&=\sp{e_1,\tau e_2,\dots, \tau e_{n-2},\tau e_{n-1},\tau e_n}.
\end{align*}
Here $e_i$ stands the element 1 in the $i$-th copy of $\Q_1$ in $V$. A general lattice in $V$ can be written as 
$$ L= \sp{ \tau^{c_1}e_1,\cdots, \tau^{c_n}e_n }$$
for some integers $c_1,\cdots,c_n$. We define its type to be $c_1+\cdots+c_n \mod n$.
Note that the group
$(1,\dots,1)\rtimes\Per(n)$ is the stabilizer of $L_0$.

Observe that under this construction the building of $\PGL_n(\Q_1)$ is exactly isomorphic to any apartment of the building attached to $\PGL_n(\Q_p)$ for a prime number $p$ and the group $\PGL_n(\Q_1)$ becomes the extended affine Weyl group of $\PGL_n(\Q_p)$.

\section{The trace formula for $\PGL_n(\Q_1)$}
Let $\Delta\subset\Z^n$ denote the subgroup spanned by the element $(1,\dots,1)$.
We write $\Lambda=\Z^n/\Delta$ and $G=\PGL_n(\Q_1)\cong (\Z^n/\Delta)\rtimes\Per(n)=\Lambda\rtimes\Per(n)$.
We denote the permutation subgroup $\Per(n)$ by $K$ and we
let $\Ga\subset G$ denote a subgroup of finite index.
The trace formula \cite{DE} for the pair $(G,\Ga)$ says that for any $f\in \ell^1(G)$ one has
$$
\sum_{\pi\in\what G} N_\Ga(\pi)\tr\pi(f)=\sum_{[\ga]}\#(\Ga_\ga\bs G_\ga)\CO_\ga(f),
$$
where $L^2(\Ga\bs G)=\bigoplus_{\pi\in\what G}N_\Ga(\pi)\pi$ is the decomposition into irreducibles, the sum on the right hand side runs over all conjugacy classes $[\ga]$ in $\Ga$, the groups $G_\ga$ and $\Ga_\ga$ are the centralizers of $\ga$ in $G$ and $\Ga$, and
$$
\CO_\ga(f)=\sum_{x\in G/G_\ga}f(x\ga x^{-1})
$$
is the \e{orbital sum}.

Consider the group $G_\R=(\R^n/\Delta(\R))\rtimes\Per(n)$.
We denote the subgroup $(\R^n/\Delta(\R))\rtimes\{1\}\cong\Lambda\otimes\R$ by $\Lambda_\R$.

Let $G^+_\R$ be the set of all $(v,p)\in G_\R$ such that $p(v)=v$.
Then $G_\R^+$ is the set of all $ak\in G_\R$ with $a\in \Lambda_\R$, $k\in K$ such that $ak=ka$.

\begin{lemma}\label{lem2.1}
The set $G^+_\R$ is a set of representatives for $G_\R$ modulo $\Lambda_\R$-conjugacy.
The set $G^+_\R$ is stable under $K$-conjugation. 
\end{lemma}

\begin{proof}
Let $(v,p)\in G_\R$ be arbitrary and $n=(x,1)\in \Lambda_\R$.
Then
$$
n(v,p)n^{-1}=(v+x-p(x),p).
$$
Let $\Eig(p,1)$ denote the eigenspace of $p$ for the eigenvalue $1$, i.e., the set of all $y\in \R^n$ such that $p(y)=y$.
We shall show that for every $v\in\R^n$ there exists $x\in\R^n$ such that $v+x-p(x)$ lies in $\Eig(p,1)$.
This then implies the same property modulo $\Delta$.
Since $p$ is an orthogonal transformation on $\R^n$, there is a $p$-stable direct sum  decomposition $\R^n=\Eig(1,p)\oplus\Eig(1,p)^\perp$ and since $1-p$ is injective on orthogonal space $\Eig(1,p)^\perp$, it is also surjective, i.e., the space $\Eig(1,p)^\perp$ equals the image of $(1-p)$.
Hence for any $v\in\R^n$ there exists $x\in\R^n$ such that $v+x-p(x)$ lies in $\Eig(p,1)$ as claimed.
Therefore the set $G_\R^+$ contains a set of representatives of $\Lambda_\R$-conjugacy.
Next assume $(v,p)\in G_\R^+$ and $a\in \Lambda_\R$ with $a(v,p)a^{-1}\in G_\R^+$.
Then $v\in \Eig(p,1)$ and if $a=(x,1)$, then $v+x-p(x)\in\Eig(p,1)$ as well, so that $x-p(x)\in\Eig(p,1)$, which implies $x-p(x)=0$, so $a(v,p)a^{-1}=(v,p)$ and $G_\R^+$ is indeed a set of representatives.
\end{proof}

Next we determine the unitary dual of $G$.
The unitary dual $\widehat \Lambda$ of $\Lambda\cong\Z^{n-1}$ is isomorphic to $\T^{n-1}$, where $\T=\{z\in\C:|z|=1\}$ is the circle group.
There is a standard way of constructing the unitary dual of a semi-direct product out of the duals of the factors.
The group $K$ acts by conjugation on $\Lambda$ and hence on its dual $\widehat \Lambda$.
For a given $\chi\in\widehat \Lambda$ let $K_\chi$ denote the stabilizer of $\chi$ in the group $K$.
For $(\sigma,V_\sigma)\in\widehat K_\chi$ define $V_{\chi,\sigma}$ to be the space of all $\phi\in L^2(K,V_\sigma)$ such that $\phi(mk)=\sigma(m)\phi(k)$ holds for all $m\in K_\chi$ and all $k\in K$.
On this space we define a unitary representation $\pi_{\chi,\sigma}$ of $G$ by
$$
\pi_{\chi,\sigma}(n,u)\phi(k)=\chi(kun(ku)^{-1})\phi(ku).
$$

\begin{proposition}
The unitary dual $\widehat G$ of $G$ is the set of all representations $\pi_{\chi,\sigma}$, where $\chi$ runs through a set of representatives of $\widehat \Lambda/K$ and $\sigma$ runs through $\widehat K_\chi$.
Note that as special case we have $\chi=1$ in which case $K_\chi=K$, so $\widehat K$ is in a canonical way a subset of $\widehat G$.
Every $\pi\in\what G$ is finite-dimensional.
\end{proposition}

\begin{proof}
Let $(\pi,V_\pi)\in\what G$.
The restriction of $\pi$ to the abelian subgroup $\Lambda$ is a direct sum of characters, so $V_\pi\cong\bigoplus_{\mu\in\hat\Lambda}V_\pi(\mu)$, where $V_\pi(\mu)$ is the set of all $v\in V_\pi$ such that $\pi(\la)v=\mu(\la)v$ holds for every $\la\in\Lambda$.
For some $\chi\in\hat\Lambda$ we have $V_\pi(\chi)\ne 0$.
If $v\in V_\pi(\chi)$ then it is easy to see that $\pi(G)v\subset \bigoplus_{\mu\in K\chi}V_\pi(\mu)$, where the sum runs over the finite $K$-orbit of $\chi$.
The isotypical space $V_\pi(\chi)$ is acted upon by the stabilizer group $K_\chi$. Let $T$ be a unitary $K_\chi$-intertwiner on this space, then $T$ extends to a unitary $\pi(G)$-intertwiner $\tilde T$ via $\tilde T(\pi(g)v)=\pi(g)Tv$.
By Schur's lemma the map $\tilde T$ is a scalar mutiple of the identity, hence so is $T$ and so, by the converse direction of Schur's lemma, the $K_\chi$-representation $\sigma$ on $V_\pi(\chi)$ is irreducible, so that finally $\pi\cong\pi_{\chi,\sigma}$.
Similarly one sees that $\pi_{\chi,\sigma}$ is irreducible.
Furthermore, $\pi_{\chi,\sigma}$ is finite-dimensional, so every $\pi\in\what G$ is finite-dimensional.

If on the other hand, $\pi_{\chi,\sigma}\cong\pi_{\chi',\sigma'}$, then $\chi'$ must lie in the $K$-orbit of $\chi$. By conjugation one may assume that $\chi=\chi'$.
Then the $K_\chi$-action on $V_{\pi_{\chi,\sigma}}(\chi)$ and $V_{\pi_{\chi,\sigma'}}(\chi)$ must be isomorphic, which means that $\sigma\cong\sigma'$.
\end{proof}

The group $K\cong\Per(n)$ acts on $\Lambda_\R\cong\R^n/\Delta(\R)$ and the closed cone
$$
\Lambda_\R^+=\{ [x_1,\dots,x_n]\in \Lambda_\R: x_1\ge x_2\ge\dots\ge x_n\}
$$
is a set of representatives of $\Lambda_\R/K$.
Define elements $\al_1,\dots,\al_{n-1}$ of the dual space by 
$$
\al_1(x)=n!(x_1-x_2),\dots,\al_{n-1}(x)=n!(x_{n-1}-x_n).
$$
Then $\Lambda_\R^+$ is the set of all $x\in \Lambda_\R$ with $\al_1(x)\ge 0,\dots,\al_{n-1}(x)\ge 0$.
Any subset $S\subset \{1,\dots,n-1\}$ defines a face of the cone $\Lambda_\R^+$ by
$$
\Lambda^+_S=\{x\in \Lambda_\R^+: \al_j(x)=0\ \Leftrightarrow\ j\in S\}.
$$
The cone $\Lambda_\R^+$ is the disjoint union of its faces, in particular, $\Lambda_{\emptyset}^+$ is the open interior of $\Lambda_\R^+$ and $\Lambda_{\{1,\dots,n-1\}}^+$ is the point $0$.

\begin{lemma}\label{lem2.3}
The set
$G_\R^+\cap (\Lambda_\R^+\times K)
$
contains a set of representatives for $G_\R$ modulo conjugation.
If $(a,k),(a',k')\in G_\R^+\cap (\Lambda_\R^+\times K)$ are conjugate, then $a=a'$ and $k'=pkp^{-1}$ for some $p\in K$ with $p(a)=a$.

Therefore, there exists unique conjugation-invariant functions $l_1,\dots,l_{n-1}$ on $G_\R$ such that
$$
l_j(v,k)=\al_j(v),\qquad\text{if}\ (v,k)\in G_\R^+\cap (\Lambda_\R^+\times K).
$$
The functions $l_1,\dots,l_{n-1}$ are integer-valued on the subgroup $G$.
\end{lemma}

\begin{proof}
First, by Lemma \ref{lem2.1}, $G_\R^+$ is a set of representatives with respect to $\Lambda_\R$-conjugation, which is $K$-stable. As every element of $G_\R^+$ is $K$-conjugate to an element of $\Lambda_\R^+\times K$, the first claim follows.
Now let $(a,k),(a',k')\in G_\R^+\cap (\Lambda_\R^+\times K)$ be conjugate, say $(a',k')=(v,p)(a,k)(v,p)^{-1}$.
Then
$$
(a',k')=(p(a)+v-pkp^{-1}(v),pkp^{-1}).
$$
Since $k(a)=a$, it follows $k'(p(a))=p(a)$, i.e., $p(a)\in\Eig(k',1)$.
As $a'\in\Eig(k',1)$ we get $v-pkp^{-1}(v)=v-k'(v)\in\Eig(k',1)$.
This can only be if $v-k'(v)=0$, so $a'=p(a)$.
But as $a,a'$ are both in $\Lambda_\R^+$, it follows that $a=a'$ as claimed.

Finally, we need to show that $l_j(x)$ is integral for $x\in G$.
Write $x=(v,p)$ with $v\in\Z^n$.
Then, as shown in the proof of Lemma \ref{lem2.1}, there exists $a\in\R^n$ such that $v+(1-p)(a)\in\Eig(p,1)$.
We have to show that this vector lies in $\frac1{n!}\Z^n$.
For this write $p$ a product of disjoint cycles, after changing the numeration one may assume
$$
p=(1,\dots,k_1)(k_1+1,\dots,k_2)\cdots(k_{r-1}+1,\dots,k_r).
$$
A basis of $\Eig(p,1)$ is given by
$$
v_1=(\underbrace{1,\dots,1}_{k_1-\text{times}},0\dots,0),\dots,v_r=(0,\dots,0,\underbrace{1,\dots,1}_{k_r-k_{r-1}-\text{times}}).
$$
We claim that one can find $a\in\R^n$ such that 
$$
v-(1-p)(a)\in \frac1{k_1}\Z v_1\oplus\dots\oplus\frac1{k_{r}-k_{r-1}}\Z v_r.
$$
We start with $v=e_1$ the first standard basis vector.
In this case the vector $a=(0,-\frac{k_1-1}{k_1},-\frac{k_1-2}{k_1},\dots,-\frac 1{k_1},0\dots,0)$ will do the job.
If $v=e_j$ for some $1\le j\le k_1$, then $v=p^{j-1}e_j$, so one can take the vector $a$ as before and apply $p^{j-1}$ to get a vector doing the job.
If $v$ is a $\Z$-linear combination of these $e_j$, then $a$ can be taken as the corresponding linear combination of the vectors one gets for the $e_j$'s.
The other coordinates are treated in the same way and the claim follows.
\end{proof}

For $u\in\C^{n-1}$ and $x\in G$ we write
$$
u^{l(x)}=u_1^{l_1(x)}\cdots u_{n-1}^{l_{n-1}(x)}.
$$

for $u\in\C^{n-1}$ let 
$$
\norm u_{\max}=\max(|u_1|,\dots,|u_{n-1}|).
$$

\begin{theorem}[Several variable Selberg type zeta function]
The infinite sum
$$
S_\Ga(u)=\sum_{[\ga]}\#(\Ga_\ga\bs G_\ga)\, u^{l(\ga)}
$$
converges localy uniformly for $\norm u_{\max}<1$ to a rational function in $u$. There exist $p_1,\dots,p_k\in\T^{n-1}$ and a polynomial $Q(u)$ such that
$$
S_\Ga(u)=\frac{Q(u)}{\prod_{i=1}^k\prod_{j=1}^{n-1}(u_j-p_{i,j})}.
$$
\end{theorem}

\begin{proof}
Let $R\subset G$ be any set of representatives of $G$ modulo conjugation.
Define a function $f_u$ on $G$ by
$$
f_u(x)=\begin{cases} u^{l(x)}& x\in R,\\
0 & x\notin R.\end{cases}
$$
Here we use the common convention that $0^0=1$.

\begin{lemma}
If $\norm u_{\max}=\max(|u_1|,\dots,|u_{n-1}|)<1$, then $f_u\in \ell^1(G)$.
\end{lemma}

\begin{proof}
By Lemma \ref{lem2.3} it suffices to show that we have
$$
\sum_{x\in\frac1{n!}\Lambda}|f_u(x)|<\infty,
$$
where $\frac1{n!}\Lambda=\frac1{n!}\Z^n/\Delta$.
Modulo the diagonal $\Delta$, every $x\in \R^n$ can be assumed to have $x_n=0$.
Then the sum is
$$
\sum_{\stack{x\in\Z^{n-1}}{x_1\ge x_2\ge\dots\ge x_{n-1}\ge 0}}|u_1|^{x_1-x_2}\cdots |u_{n-2}|^{x_{n-2}-x_{n-1}}|u_{n-1}|^{x_{n-1}}.
$$
If all $|u_j|$ are $\le q$ for some $0<q<1$, then each summand is less that $q^{x_1}$. So we have to show that for $q<1$
one has
$
\sum_{k=0}^\infty c_k q^k<\infty,
$
where $c_k$ is the number of tuples  of integers with $k\ge x_2\ge\dots\ge x_{n-1}\ge 0$ which is $\le (k+1)^{n-1}$, whence the claim. 
\end{proof}

We want to plug $f_u$ into the trace formula.
As either side of the trace formula is invariant under conjugation, neither depends on the choice of $R$.
We give a special choice for our computations.

Let $S$ be a partition of $n$, i.e., $S$ is a tuple $(n_1,\dots,n_r)$ of natural numbers with 
$n=n_1+\dots+n_r$.
Then set
$$
\Lambda_S^+=\left\{ \(t_1e(n_1),\dots,t_re(n_r)\)\in\R^n/\Delta: t_1>t_2>\dots>t_r\right\},
$$
where $e(m)=(1,\dots,1)\in\R^m$.
For $x\in \Lambda_S^+$ we write $x=[t_1,\dots,t_r]_S$ for these coordinates.
We define
$$
\Lambda_{S,\frac1n\Z}^+=\left\{\left[\frac{k_1}{n_1},\dots,\frac{k_r}{n_r}\right]_S:k_1,\dots,k_r\in\Z\right\}.
$$ 
Let $K_S$ denote the pointwise stabilizer of $\Lambda_S^+$ in $K$ and fix a set $R_S\subset K_S$ of representatives of $K_S$ modulo conjugation.

If $E\subset F$ is an extension of groups, the $F$-conjugacy class of an element $e$ of $E$ intersected with $E$, may decompose into several $E$-conjugacy classes.
In the following Lemma we show that this is not the case for the extension $G\subset G_\R$.

\begin{lemma}\label{lem2.5}
Every $x\in G$ is in $G_\R$ conjugate to a unique element of 
$$
\bigcup_S \Lambda_{S,\frac1n\Z}^+\times R_S.
$$
Every element of this set is $G_\R$-conjugate to an element of $G$.
If $x,y\in G$ are $G_\R$-conjugate, then they are $G$-conjugate.
\end{lemma}

\begin{proof}
The first two statements are proven similar to Lemma \ref{lem2.3}.
For the third suppose $x=(a,k)$ and $y=(a',k')$ are 
$G_\R$-conjugate, i.e., there exists $(v,p)\in G_\R$ 
with $(a',k')=(v+p(a)-pkp^{-1}(v),pkp^{-1})$.
Then $v-k'(v)$ lies in the intersection of $\Eig(k',1)^\perp$ and $\Z^n$.
Writing $p$ as a product of disjoint cycles as in the proof of Lemma \ref{lem2.3} one sees that there exists $w\in\Z^n$ such that $v-k'(v)=w-k'(w)$.
\end{proof}

Let now $\pi\in\what G$.
Then $\pi$ is finite-dimensional and
$$
\tr\pi(f_u)=\sum_{x\in R}u^{l(x)}\tr\pi(x).
$$
Since the functions $l_1,\dots,l_{n-1}$ are defined on $G_\R$ and are conjugation-invariant, we may, in the computation assume that $R$ is equal to the set in Lemma \ref{lem2.5}, although this set is not contained in $G$.
In the expression $\pi(x)$ for $x\in R\smallsetminus G$ we then replace $x$ with any $G_\R$-conjugate inside $G$.
We then compute
$$
\tr\pi(f_u)=\sum_S\sum_{a\in \Lambda_{S,\frac1n\Z}^+}u^{l(a)}\sum_{k\in R_S}\tr\pi(ak).
$$
Let $V_\pi= V_{\pi,1}\oplus\dots\oplus V_{\pi,m}$ be the decomposition into  $\Lambda$-eigenspaces, i.e., each $a\in \Lambda$ acts on $V_{\pi,j}$ as multiplication by $\la_j^a$ for some character $\Lambda\to\T$, $a\mapsto \la_j^a$, where $\T$ is the circle group, i.e., the set of complex numbers of absolute value one.
Then $\tr\pi(f_u)$ equals
$$
\sum_S\sum_{j=1}^m\sum_{a\in \Lambda_{S,\frac1n\Z}^+}u^{l(a)}\la_j^{a}
\underbrace{\sum_{k\in R_S}\tr\(\pi(k)|V_{\pi,j}\)}_{=\mu_j}
$$

\begin{lemma}
Let $V$ denote a $\Q$ vector space of dimension $r\in\N$.
Let $V_\R=V\otimes\R$ and let $C\subset V_\R$ be an open rational sharp cone with $r$ sides, i.e., its closure $\ol C$ does not contain a line and there exist $\al_1,\dots,\al_r\in\Hom(V,\Q)$ such that
$$
C=\{ v\in V_\R: \al_1(v)>0,\dots,\al_r(v)>0\}.
$$
Let $\Sigma\subset V$ be a lattice, i.e., a finitely generated subgroup which spans $V$.
Then there exists a finite subset $F\subset\Sigma$ and $a_1,\dots,a_r\in\Sigma$ such that $C\cap\Sigma$ is the set of all $v\in V$ of the form
$$
v=v_0+k_1a_1+\dots+k_ra_r,
$$
where $v_0\in F$ and $k_1,\dots,k_r\in\N_0$.
The vector $v_0$ and the numbers $k_1,\dots,k_r\in\N_0$ are uniquely determined by $v$.
\end{lemma}

\begin{proof}
For $j=1,\dots,r$ let $a_j\in\Sigma$ be the unique element such that $\al_i(a_j)=0$ for $i\ne j$ and $\al_j(a_j)$ is $>0$ and minimal.
Then $a_1,\dots,a_r$ is a basis of $V$ inside $\Sigma$, hence it generates a sublattice $\Sigma'\subset \Sigma$.
Let $F$ be a set of representatives of $\Sigma/\Sigma'$ which may be chosen such that each $v_0\in F$ lies in $C$, but for every $j=1,\dots,r$ the vector $v_0-a_j$ lies outside $C$.
It is clear that every $v$ of the form given in the lemma is in $C\cap\Sigma$.

For the last statement, let $v\in C\cap \Sigma$. Then there are uniquely determined $v_0\in F$, $k_1,\dots,k_r\in\Z$ such that $v=v_0+k_1a_1+\dots+k_ra_r$.
We have to show that $k_1,\dots,k_r\ge 0$.
Assume that $k_j<0$.
Then
$$
0<\al_j(v)=\al_j(v_0)+k_j\al_j(a_j)\le\al_j(v_0)-\al_j(a_j)=\al_j(v_0-a_j)
$$
and the latter is $\le 0$, as $v_0-a_j$ lies outside $C$, a contradiction!
\end{proof}

We apply this lemma to $V$ being the $\Q$-span of $\Lambda\cap \Lambda_S^+$ and $C=\Lambda_S^+$.
We find that $\tr\pi(f_u)$ equals
\begin{align*}
&\sum_S\sum_{j=1}^m\mu_j \sum_{a_0\in F}\sum_{k_1,\dots,k_r=0}^\infty u^{l(a_0+k_1a_1+\dots+k_ra_r)}\la_j^{a_0+k_1a_1+\dots+k_ra_r}\\
&= \sum_S\sum_{j=1}^m\mu_j \sum_{a_0\in F}u^{l(a_0)}\la_j^{a_0}\frac1{1-\la_j^{a_1}u^{l(a_1)}}\dots\frac1{1-\la_j^{a_r}u^{l(a_r)}},
\end{align*}
where in the last row all sums are finite.
We have shown

\begin{lemma}\label{lem2.8}
For each $\pi\in\what G$, the map $u\mapsto\tr\pi(f_u)$, defined for small $u$, is a rational function in $u$.
\end{lemma}

We now finish the proof of the theorem.
For $\norm u_{\max}<1$ the function $f_u$ goes into the trace formula. This in particular means that the sum
$$
\sum_{[\ga]}\#(\Ga_\ga\bs G_\ga)\CO_\ga(f_u)=S_\Ga(u)
$$
converges locally uniformly.
As the quotient $\Ga\bs G$ is finite, the space $L^2(\Ga\bs G)$ is finite-dimensional, so the sum $\sum_{\pi\in\what G}N_\Ga(\pi)\tr\pi(f_u)$ is a finite sum, i.e., the coefficient $N_\Ga(\pi)$ vanishes for almost all $\pi$.
So $S_\Ga(u)$ is a finite sum of rational functions of the form in Lemma \ref{lem2.8}.
As the representation $\pi$ is unitary, the complex numbers $\la_1,\dots,\la_{m}$ in the lemma are all in $\T$.
The proof of the theorem is finished.
\end{proof}

\section{Geometric interpretation}
In this section we assume  $\Ga$ to be torsion-free.
It follows that  $\Ga$ is the fundamental group of $\Ga\bs\CB$, where $\CB\cong\R^{n-1}$ is the building of $\PGL_n(\Q_1)$.
Then each conjugacy class $[\ga]$ gives a homotopy class of loops in $\Ga\bs\CB$, where a loop is a continuous map $S^1\to\CB$.
The euclidean structure  makes $\CB$ and $\Ga\bs\CB$ a Riemannian manifold, where the geodesics in $\CB$ are straight lines.

\begin{lemma}
Every loop on $\Ga\bs\CB$ is homotopic to a closed geodesic.
\end{lemma}

This is a well known property of compact Riemann manifolds, see for instance \cite{Gallot} 2.98.
For the convenience of the reader we include a proof in this special situation, which also works if $\Ga\bs\CB$ is non-compact and introduces some notation which will be needed later.

\begin{proof}
Take a loop in $\Ga\bs\CB$ and lift it to a continuous path on $\CB$ of infinite length, which is closed by some non-trivial $\ga\in\Ga$. This means that the path is $\ga$-stable and $\ga$ acts on it by translation, i.e., the path is given by a continuous $c:\R\to\CB$ such that $\ga c(t)=c(t+1)$.
The element $\ga$ acts on $\CB\cong\R^{n-1}$ as an affine motion.
The affine subspace $P_\ga=\{ x\in\CB: d(x,\ga x)\text{ is minimal}\}$ is $\ga$-stable and $\ga$ acts as a translation by some vector in $P_\ga$ followed by a linear map on the orthogonal space $P_\ga^\perp$.
From this it becomes clear that $c$ may, modulo homotopy through $\ga$-stable maps, be assumed affine and then moved inside $P_\ga$, again by a $\ga$-stable homotopy.
As $\ga$ is a translation on $P_\ga$, the claim follows.
\end{proof}

The group $\Ga$ is called a \e{translation group}, if $\Ga\subset \Lambda$.
Since $\Lambda$ has finite index in $G$, every $\Ga$ contains a finite-index translation subgroup.

We say that a geodesic in $\CB$ or $\Ga\bs\CB$ is a \e{rational geodesic}, if it is contained in the 1-skeleton $\CB_1$ or $(\Ga\bs \CB)_1$.

\begin{lemma}
If the homotopy class attached to a given class $[\ga]$ contains a rational geodesic then one has $l_j(\ga)=0$ for all but one $j\in\{ 1,\dots,n-1\}$.

Conversely, if $l_j(\ga)=0$ for all but one $j\in\{ 1,\dots,n-1\}$, then the homotopy class attached to some power $\ga^k$ of $\ga$ contains a rational geodesic.
The minimal number $k$ as above  is $\le n$ and if $\Ga$ is a translation group, one always has $k=1$.
\end{lemma}

\begin{proof}
Suppose that the homotopy class given by $[\ga]$ contains a rational geodesic $c$ in $\Ga\bs\CB$.
This means that there exists a lift $\tilde c\in\CB_1$  and $\ga$ induces a translation on $\tilde c$. 
Since $\ga$ preserves the affine structure on $\CB$, we may choose the origin in a vertex on the rational geodesic $\tilde c$, or, which amounts to the same, conjugate $\ga$ by an element on $\Lambda$.
Then $\ga$ maps $x\in\CB$ to $Fx+t$, where $F$ is linear with $F(t)=t$, so $\ga\in G_\R^+$.
Conjugating by some $k\in K$, we may assume $t\in \Lambda_\R^+$.
As $\tilde c$ is rational, $t$ lies in a $1$-dimensional face of $\Lambda_\R^+$, which is equivalent to saying $l_j(\ga)=0$ for all but one $j$.

For the converse direction, assume $\ga x=F(x)+t$ with $F(t)=t$.
Let $k$ be the order of $F$ on $K=\Per(n)$, then $k$ is a divisor of $n$ and $\ga^k$ is a translation.
As $t$ lies in a one-dimensional face of $\Lambda_\R^+$, $\ga$ closes a rational geodesic.
\end{proof}

\section{An Ihara type formula}\label{SecIhara}
The Ihara zeta function of a finite, $q+1$ regular graph $X$ is defined as the infinite product
$$ 
Z_X(u)=\prod_{c}(1-u^{l(c)})^{-1},
$$
where the product runs over all backtrackingless closed cycles $c$ and $l(c)$ denotes the length.
The Ihara formula asserts that when $|u|$ is small enough, $Z_X(u)^{-1}$ converges to a polynomial so that
$$
Z_X(u)^{-1}=\det(1-Au+qu^2)(1-u^2)^{-\chi},
$$
where $A$ is the adjacency operator and $\chi$ is the Euler-characteristic (which is always non-positive).
It has been proven in ascending order of generality in 
\cites{Ihara,Hashimoto,Sunada,Bass}.
For higher dimensional buildings, the question for a generalized Ihara formula is still open. For the $\PGL_3$-case see \cites{KL,KangLi}.

Fix the following set of generators of $\Lambda$, 
$$ 
S=\{[a_1,\cdots,a_n] \in \Lambda, \max_{1\leq i,j \leq n}\{|a_i-a_j|\}=1 \}.
$$
As the set $S$ has $2^n-2$ elements and the group $\Lambda$ is infinite, the Cayley graph $X$ of $(\Lambda,S)$ is a $(2^n-2)$-regular infinite graph. We shall show that the Cayley graph $X$ is isomorphic to the $1$-skeleton of the building $\CB$ of  $\PGL_n(\Q_1)$ as graphs.

Recall that $\PGL_n(\Q_1) = \Lambda \rtimes \Per(n)$ acts transitively on  the set of vertices $\CB_0$ of its building and $\Per(n)$ is the stabilizer of the vertex $[L_0]$ where $L_0$ is the lattice spanned by the standard basis $e_1,\cdots,e_n$. Therefore, for $[a_1,\cdots,a_n] \in \Lambda$, 
$$ [a_1,\cdots,a_n] \mapsto [\langle \tau^{a_1}e_1,\cdots, \tau^{a_n}e_n \rangle]$$ defines a bijection between $\Lambda$ and $\CB_0$. Moreover, by definition, two vertices  $[\langle \tau^{a_1}e_1,\cdots, \tau^{a_n}e_n \rangle]$ and 
$[\langle \tau^{b_1}e_1,\cdots, \tau^{b_n}e_n \rangle]$
 form an edge if there is some integer $t$ so that 
 $a_i \leq b_i+t \leq a_i+1$ for all $i$, which is equivalent to that $[a_1,\cdots,a_n]$ and $[b_1,\cdots,b_n]$ differ by some element of $S$. Therefore, the above bijection indeed gives a graph isomorphism between the Cayley graph $X$ and the 1-skeleton $\CB_1$ of $\CB$. For convenience, we identify $X$ with $\CB_1$ in the rest of the paper.

Now for $a=[a_1,\cdots,a_n] \in \Lambda$, define its type $\tau(a)$ as $a_1+\cdots+a_n \mod n$. 
Fix a subgroup $\Gamma$ of $\Lambda$ of finite index $N$. Then the quotient of $X$ by $\Gamma$, denoted by $X_\Gamma$, is the Cayley graph $(\Lambda/\Gamma,S)$ of $N$ vertices; it can be considered
as the 1-skeleton of the $(n-1)$-dimensional simplicial torus $\CB/\Ga$, denoted by $\CB_\Ga$. We shall assume that $\Gamma$ only contains type zero elements so that the type is well-defined on $\Lambda/\Gamma$. 

In this case, the Ihara formula states that
$$ 
Z(X_\Gamma, u)^{-1} = \det(I_N-Au+(2^n-3)u^2I_N)(1-u^2)^{-\chi(X_\Gamma)}$$
which only encodes the information of the 1-skeleton $ X_\Ga$ but not the whole complex $\CB_\Ga$. Therefore, we shall study a different kind of zeta function which encodes more information about $\CB_\Ga$.

For a function $f:\Lambda/\Gamma \to \C$, define 
$$ A_i f( g \Gamma )= \sum_{s \in S, \tau(s)=i} f(s g \Gamma),$$
call $A_i$ the type $i$ adjacency operator of $X_\Gamma$. 


A path $(v_0,\cdots,v_n)$ on the graph $X$ is called a \emph{geodesic}, if it is a straight line in $\R^n$. As there are only finitely many directions of the latter which lie in the 1-skeleton, if we defined a zeta function as the product over all closed geodesics in $X_\Ga$, this zeta function would be a finite product of functions, each of which is attached to one direction.
We instead consider a zeta function corresponding to one direction only.
A geodesic $(v_0,\dots,v_n)$  is called a  \emph{positive geodesic} if it is geodesic in $\R^n$ and $\tau(v_{i+1})=\tau(v_i)+1$ for $i$.
A closed path $c$ in $X_\Gamma$ is a positive geodesic if its lifting in $X$ is a positive geodesic; $c$ is primitive if $c$ is not equal to a shorter path repeated several times.
Two closed paths in $X_\Gamma$ are equivalent if one can be obtained from the other by changing the starting vertex. Denote the equivalence class of a closed path $c$ by $[c]$.
In this paper, we shall consider the following zeta function on $X_\Gamma$
$$ Z_+(u)=Z_{+}(X_\Gamma, u) = \prod_{[c]} (1-u^{l(c)})^{-1}
$$
where $[c]$ runs through all equivalence classes of primitive positive closed geodesics in $X$. 
We consider this particular zeta function, because it is the smallest building block of a zeta function counting closed geodesics in the 1-skeleton and because of the following results which generalize Ihara's Identity, relate this zeta function to Langlands L-functions and to the group-theoretic several variable function $S_\Ga$.

First, we will show that
\begin{theorem} \label{theorem1}
The infinite product $Z_{+}(X_\Gamma,u)^{-1}$ converges to a polynomial which can be expressed as 
$$ Z_{+}(X_\Gamma, u)^{-1} = \det(I_N -A_1u +\cdots+ (-1)^{n-1}A_{n-1}u^{n-1}+(-1)^n u^n I_N).$$
\end{theorem}
\begin{remark}
As the Euler characteristic of the torus $\CB_\Ga$ is zero, we can rewrite the above equation as 
$$ 
Z_{+}(X_\Gamma, u)^{-1} = \frac{\det(I_N-A_1u +\cdots+ (-1)^{n-1}A_{n-1}u^{n-1}+(-1)^n u^n I_N)}{(1-u^n)^{\chi(\CB_\Ga)}}.
$$
\end{remark}

Now given a character 
$\rho:\Lambda\to\C^\times$, define the \emph{Satake parameters} of $\rho$ to be $\rho_j=\rho(e_j)$, where $e_1,\dots,e_n$ is the standard basis of $\Z^n$.
Then $\rho_1\cdots\rho_n=1$ and we define the \emph{Langlands $L$-function} of $\rho$ to be
$$
L(\rho,u)=\prod_{j=1}^n(1-\rho_ju)^{-1}.
$$
Finally, we define the $L$-function of $\Lambda/\Gamma$ as
$$
L(\Lambda/\Gamma,u)=\prod_{\rho\in\widehat{\Lambda/\Gamma}}L(\rho,u).
$$
Note that there is no multiplicity showing up as an exponent to the factor $L(\rho,u)$ as, since $\Lambda/\Gamma$ is an abelian group, characters do not come with mutliplicities other than one.

\begin{theorem} \label{theorem2}
$Z_{+}(X_\Gamma,u) = L(\Lambda/\Gamma,u).$
\end{theorem}

\begin{proof}[Proof of Theorem \ref{theorem1} and \ref{theorem2}]
Let $s_i$ be an element of $S$ which has a representative in $\Z^n$ with all coordinates equal to zero except the $i$-th coordinate equal to 1. Set $S_0=\{s_1,\cdots,s_n\},$ then 
$$ S=\left\{ \sum_{s \in S'} s: \mbox{$S'$ is a proper subset of $S_0$} \right\}.$$
Given a vertex $g+\Gamma$ in $X_\Gamma$, each positive geodesic has the form 
$$ g+\Gamma \to (s_i + g)+ \Gamma  \to  (2 s_i +g)+\Gamma \to \cdots.$$
and it is primitive if its length is equal to the order of $s_i$ in $\Lambda/\Gamma$. Denote this order by $m_i$, then the contribution of such kind of positive closed geodesics to the Euler product $Z_+(X_\Ga,u)$ is given by
$$ (1-u^{m_i})^{\frac{N}{m_i}} = \det\left(I_N - \lambda(s_i)u\right).$$
Here $\lambda$ is the  regular representation of $\Lambda/\Gamma$. We conclude that 
$$ Z_+(X_\Gamma,u)^{-1}=\prod_{i=1}^n \det (I_N- \lambda(s_i) u).$$
The right hand side is clearly equal to $L(\Lambda/\Gamma,u)^{-1}$.
On the other hand,
$$ A_i = \sum_{S'\subset S_0, |S'|=i} \lambda \left(\sum_{s \in S'}s \right) = 
\sum_{S'\subset S_0, |S'|=i} \prod_{s \in S'} \lambda (s ),$$
so that
$$
I_N -A_1u +\cdots+ (-1)^{n-1}A_{n-1}u^{n-1}+(-1)^n u^n I_N = \prod_{i=1}^n (I_N- \lambda(s_i) u),
$$
which completes the proof of the two theorems.
\end{proof}

\section{Comparison}

We shall now compare the different types of zeta functions in the following theorem.

\begin{theorem}
If  $\Ga$ is a translation group, then
$$
S_\Ga(x,0,\dots,0)=(n-1)!\frac{Z_+'}{Z_+}(x).
$$
\end{theorem}

\begin{proof}
Let $N$ be the index of $\Ga$ in $\Lambda$.
Then $N$ equals the number of vertices in $\Ga\bs\CB$.
As $\Ga\subset\Lambda$  we have for every $\ga\in\Ga$ that $G_\ga =\Lambda\rtimes K_\ga$ and so $\#(\Ga_\ga\bs G_\ga)=N\# K_\ga$.
In the sum $S_\Ga(u)=N\sum_{[\ga]}\#K_\ga u^{l(\ga)}$ we find that if $u=(x,0,\dots,0)$ there will only survive those summands with $l_2(\ga)=\dots=l_{n-1}(\ga)=0$, i.e., those $\ga$ which are in $G$ conjugate to an element of the form $(c,0,\dots,0)$ for some $c>0$.
The $K$-centralizer $K_\ga$ of such an element is isomorphic to $\Per(n-1)$, hence $\#K_\ga=(n-1)!$.
Such a $\ga$ closes a geodesic $c$ in the 1-skeleton and the number of vertices in that geodesic equals $l(\ga_0)$, where $\ga_0$ is the underlying primitive element.
We also write $c_0$ for the underlying primitive geodesic.
The union of all geodesics inside the 1-skeleton of $\Ga\bs\CB$ which are homotopic to $c$ contains all vertices of $\Ga\bs\CB$, hence, if $k$ is their number, one has
$N=kl(\ga_0)=kl(c_0)$, or
$$
S_\Ga(x,0\dots,0)=(n-1)!\sum_{c}l(c_0)x^{l(c)},
$$
where the sum runs over all positive closed geodesics in $\Ga\bs \CB$.
On the other hand one has the following identities of formal power series,
\begin{align*}
\frac{Z_+'}{Z_+}(x)&= \(-\log\(\prod_{c_0}(1-x^{l(c_0)})\)\)'\\
&=\(\sum_{c_0}\sum_{j=1}^\infty\frac{x^{l(c_0)j}}j\)'\\
&=\sum_{c_0}l(c_0)\sum_{j=1}^\infty x^{l(c_0)j}=\sum_{c}l(c_0)x^{l(c)}.
\tag*\qedhere
\end{align*}
\end{proof}

\end{document}